\documentclass[11pt]{article}

\usepackage{amsmath, amssymb, amsthm}
\usepackage{mathrsfs}
\usepackage{hyperref}
\usepackage{geometry}
\hypersetup{
    colorlinks=true,
    linkcolor=blue,
    citecolor=blue,
    urlcolor=blue
}

\title{\textbf{Rigorous Geometric Obstructions for Fourier Curves Generated by Prime Numbers}}

\author{
Dimitris Vartziotis$^{1,2,*}$
}

\date{}

\theoremstyle{plain}
\newtheorem{theorem}{Theorem}[section]
\newtheorem{lemma}[theorem]{Lemma}
\newtheorem{proposition}[theorem]{Proposition}
\newtheorem{definition}[theorem]{Definition}

\newcommand{\N}{\mathbb{N}}
\newcommand{\R}{\mathbb{R}}

\DeclareMathOperator{\Real}{Re}
\DeclareMathOperator{\Imag}{Im}

\DeclareMathOperator{\diam}{diam}

\begin{document}

\maketitle

\vspace{-1em}

\begin{center}
\footnotesize
$^{1}$ NIKI -- Digital Engineering, Ioannina, Greece\\
$^{2}$ TWT Science \& Innovation, Stuttgart, Germany\\
$^{*}$ Corresponding author: \texttt{dimitris.vartziotis@nikitec.gr}
\end{center}

\vspace{1em}

\begin{abstract}
We study planar curves defined by finite Fourier series of the form 
\( F_n(t)=\sum_{p\le n} v_p(n!)\, e^{i p t} \), 
where the frequencies are the prime numbers and \( v_p(n!) \) denotes the exponent of the prime \(p\) in the factorization of \(n!\). We establish several rigorous obstructions to uniform geometric regularity as \( n\to\infty \). In particular, we prove that the curve lengths grow without bound, that neither the first nor the second derivatives remain uniformly bounded, and that the diameters grow at least on the order of \( n\log\log n \). As a consequence, the covering numbers of the curves satisfy explicit quantitative lower bounds. These results provide a rigorous explanation for the complex geometric behavior observed in numerical investigations of this model.
\end{abstract}

\vspace{1em}



\section{Introduction}

Fourier series whose frequency sets encode arithmetic structure play a central role in analytic number theory, most notably in the study of exponential sums over primes (see \cite{HardyLittlewood1923,Montgomery1994}). Fourier series with prime frequencies have also been explored as a geometric way to represent arithmetic structure through planar curves (see \cite{Vartziotis2026a,Vartziotis2026b} for recent constructions and experimental investigations).

Numerical investigations reveal curves with complicated behavior at many scales. However, visual complexity alone does not constitute mathematical evidence. A natural question therefore arises:

\medskip
\noindent\emph{Can the observed geometric irregularity of Fourier curves with prime frequencies be justified rigorously, independently of discretization or heuristic interpretation?}
\medskip

Rather than attempting to compute a fractal dimension, we establish several analytic statements that force a lack of uniform geometric regularity as $n\to\infty$. The analysis uses standard tools from harmonic analysis (Parseval's identity) and elementary estimates from number theory (factorial valuations and asymptotics for primes).

For related constructions linking spectra and geometry in polygonal transformations, see \cite{VartziotisWipper2010}.

\medskip
\noindent\textbf{Series context.}
This paper forms the third part of a series.
The first paper \cite{Vartziotis2026a} introduced the Fourier model and its arithmetic motivation.
The second paper \cite{Vartziotis2026b} presented an experimental study of the geometric behavior.
The present paper provides rigorous analytic obstructions explaining why uniform smooth geometric behavior cannot persist as $n\to\infty$.

\section{Definitions and preliminaries}

For each $n\in\N$, define the finite Fourier series
\begin{equation}\label{eq:Fn}
F_n(t)=\sum_{p\le n} v_p(n!)\, e^{i p t},
\qquad t\in[-\pi,\pi],
\end{equation}
where $v_p(n!)$ is the exponent of $p$ in $n!$, given by Legendre's formula
\begin{equation}\label{eq:legendre}
v_p(n!)=\sum_{k\ge1}\left\lfloor\frac{n}{p^k}\right\rfloor.
\end{equation}
Related additive prime factor constructions were studied in \cite{VartziotisTzavellas2016}.

We associate the planar curve
\begin{equation}\label{eq:curve}
\Gamma_n(t)=\big(\Real F_n(t),\,\Imag F_n(t)\big)\subset\R^2.
\end{equation}
Since $F_n$ is a trigonometric polynomial, each $\Gamma_n$ is a smooth (hence rectifiable) curve for fixed $n$.
Our focus is the behavior of this family as $n\to\infty$.

\begin{lemma}\label{lem:vp-bounds}
For every prime $p\le n$,
\[
\frac{n}{p}-1 \le v_p(n!) \le \frac{n}{p-1}.
\]
\end{lemma}

\begin{proof}
The upper bound follows from \eqref{eq:legendre} using $\lfloor n/p^k\rfloor\le n/p^k$ and summing the geometric series:
\[
v_p(n!)\le \sum_{k\ge1}\frac{n}{p^k}=\frac{n}{p-1}.
\]
The lower bound follows from keeping only the first term:
$v_p(n!)\ge \lfloor n/p\rfloor\ge n/p-1$.
\end{proof}

\section{Growth of derivatives}

Differentiating termwise,
\[
F_n'(t)=\sum_{p\le n} i\,p\,v_p(n!)\, e^{i p t},
\qquad
F_n''(t)=-\sum_{p\le n} p^2 v_p(n!)\, e^{i p t}.
\]

By orthogonality of the exponentials, Parseval's identity gives
\begin{equation}\label{eq:parseval1}
\|F_n'\|_{L^2([-\pi,\pi])}^2
=
\int_{-\pi}^{\pi}|F_n'(t)|^2 dt
=2\pi\sum_{p\le n} p^2 v_p(n!)^2,
\end{equation}
\begin{equation}\label{eq:parseval2}
\|F_n''\|_{L^2([-\pi,\pi])}^2
=
\int_{-\pi}^{\pi}|F_n''(t)|^2 dt
=2\pi\sum_{p\le n} p^4 v_p(n!)^2.
\end{equation}

\begin{proposition}\label{prop:l2-diverge}
As $n\to\infty$, both $\|F_n'\|_{L^2([-\pi,\pi])}$ and $\|F_n''\|_{L^2([-\pi,\pi])}$ diverge.
Moreover, $\|F_n''\|_{L^2}$ diverges faster than $\|F_n'\|_{L^2}$.
\end{proposition}

\begin{proof}
For primes $p\in(n/3,n/2]$ we have $\lfloor n/p\rfloor=2$ and $p^2>n$ (for $n$ large), hence $v_p(n!)=\lfloor n/p\rfloor=2$.
Therefore
\[
\sum_{p\le n} p^2 v_p(n!)^2
\;\ge\;
\sum_{n/3<p\le n/2} p^2\cdot 4
\;\ge\;
4\left(\frac{n}{3}\right)^2\big(\pi(n/2)-\pi(n/3)\big),
\]
and similarly,
\[
\sum_{p\le n} p^4 v_p(n!)^2
\;\ge\;
4\left(\frac{n}{3}\right)^4\big(\pi(n/2)-\pi(n/3)\big),
\]
where $\pi(x)$ denotes the number of primes $\le x$.
By the prime number theorem (see \cite{Montgomery1994}), $\pi(n/2)-\pi(n/3)\to\infty$ and is comparable to $n/\log n$, so both sums diverge and the second has higher order growth.
Using \eqref{eq:parseval1}--\eqref{eq:parseval2} yields the claim.
\end{proof}

\section{Divergence of curve length (no uniform length bound)}\label{sec:length}

\begin{definition}
The length of $\Gamma_n$ on $[-\pi,\pi]$ is
\[
L(\Gamma_n)=\int_{-\pi}^{\pi}|\Gamma_n'(t)|\,dt=\int_{-\pi}^{\pi}|F_n'(t)|\,dt.
\]
\end{definition}

\begin{lemma}\label{lem:l1-l2-linf}
For any measurable $2\pi$-periodic function $f$,
\[
\|f\|_{L^1([-\pi,\pi])}\ge \frac{\|f\|_{L^2([-\pi,\pi])}^2}{\|f\|_{L^\infty([-\pi,\pi])}}.
\]
\end{lemma}

\begin{proof}
We have $\|f\|_2^2=\int |f|^2 \le \|f\|_\infty\int |f|=\|f\|_\infty\|f\|_1$.
\end{proof}

\begin{theorem}\label{thm:length}
The lengths $L(\Gamma_n)$ are not uniformly bounded as $n\to\infty$. In fact, $L(\Gamma_n)\to\infty$.
\end{theorem}

\begin{proof}
By Lemma \ref{lem:l1-l2-linf} with $f=F_n'$,
\[
L(\Gamma_n)=\|F_n'\|_{L^1}\ge \frac{\|F_n'\|_{L^2}^2}{\|F_n'\|_{L^\infty}}.
\]
Using the triangle inequality,
\[
\|F_n'\|_{L^\infty}
\le \sum_{p\le n} p\,v_p(n!).
\]
By Lemma \ref{lem:vp-bounds}, $p\,v_p(n!)\le p\cdot \frac{n}{p-1}\le 2n$ for all primes $p\ge 2$, hence
\[
\|F_n'\|_{L^\infty}\le 2n\,\pi(n).
\]
On the other hand, Proposition \ref{prop:l2-diverge} gives $\|F_n'\|_{L^2}^2\to\infty$.
More quantitatively, the proof of Proposition \ref{prop:l2-diverge} yields
\[
\|F_n'\|_{L^2}^2
=2\pi\sum_{p\le n} p^2 v_p(n!)^2
\ge c_1\,n^2\big(\pi(n/2)-\pi(n/3)\big)
\]
for some constant $c_1>0$ and all large $n$.
Combining these bounds,
\[
L(\Gamma_n)\ge
\frac{c_1\,n^2(\pi(n/2)-\pi(n/3))}{2n\,\pi(n)}
=
c_2\,n\,\frac{\pi(n/2)-\pi(n/3)}{\pi(n)}.
\]
By the prime number theorem, the ratio $\frac{\pi(n/2)-\pi(n/3)}{\pi(n)}$ stays bounded away from $0$ for all large $n$ (see \cite{Montgomery1994}), so $L(\Gamma_n)\gtrsim n\to\infty$.
\end{proof}

\section{Failure of uniform $C^1$ bounds}\label{sec:C1}

\begin{theorem}\label{thm:C1}
There is no subsequence of $\{\Gamma_n\}$ whose first derivative is uniformly bounded on $[-\pi,\pi]$.
Equivalently, $\|F_n'\|_{L^\infty([-\pi,\pi])}\to\infty$.
\end{theorem}

\begin{proof}
Since $\|f\|_{L^2}\le \sqrt{2\pi}\,\|f\|_{L^\infty}$, we have
\[
\|F_n'\|_{L^\infty}\ge \frac{1}{\sqrt{2\pi}}\|F_n'\|_{L^2}.
\]
By Proposition \ref{prop:l2-diverge}, $\|F_n'\|_{L^2}\to\infty$, hence $\|F_n'\|_{L^\infty}\to\infty$.
\end{proof}

\section{Failure of uniform $C^{1,1}$ bounds}\label{sec:C11}

\begin{theorem}\label{thm:C11}
There is no subsequence of $\{\Gamma_n\}$ with uniformly bounded second derivative on $[-\pi,\pi]$.
Equivalently, $\|F_n''\|_{L^\infty([-\pi,\pi])}\to\infty$.
\end{theorem}

\begin{proof}
As in Theorem \ref{thm:C1},
\[
\|F_n''\|_{L^\infty}\ge \frac{1}{\sqrt{2\pi}}\|F_n''\|_{L^2}.
\]
Proposition \ref{prop:l2-diverge} implies $\|F_n''\|_{L^2}\to\infty$, hence $\|F_n''\|_{L^\infty}\to\infty$.
\end{proof}

\section{Diameter growth and covering lower bounds}\label{sec:cover}

Let $N_n(\varepsilon)$ be the minimal number of closed axis-aligned squares of side length $\varepsilon$ needed to cover the set $\Gamma_n([-\pi,\pi])$.

\begin{lemma}\label{lem:connected-cover}
Let $S\subset\R^2$ be connected and covered by $m$ sets, each of diameter at most $\delta$.
Then $\diam(S)\le m\,\delta$.
\end{lemma}

\begin{proof}
Let $x,y\in S$ with $|x-y|=\diam(S)$.
Since $S$ is connected, the subcollection of covering sets that intersect $S$ has a connected union (otherwise $S$ would be separated into two disjoint open sets in the relative topology).
Thus there exists a chain of covering sets $A_1,\dots,A_m$ such that $x\in A_1$, $y\in A_m$, and $A_j\cap A_{j+1}\neq\emptyset$ for each $j$.
Pick points $z_j\in A_j\cap A_{j+1}$.
Then
\[
|x-y|
\le |x-z_1|+\sum_{j=1}^{m-2}|z_j-z_{j+1}|+|z_{m-1}-y|
\le m\,\delta,
\]
because each term connects two points lying in one set of diameter $\le\delta$.
\end{proof}

\begin{theorem}\label{thm:cover}
There exists $c>0$ such that for all sufficiently large $n$ and all $0<\varepsilon\le 1$,
\[
N_n(\varepsilon)\ge c\,\frac{n\log\log n}{\varepsilon}.
\]
\end{theorem}

\begin{proof}
Each square of side $\varepsilon$ has diameter $\sqrt{2}\,\varepsilon$.
By Lemma \ref{lem:connected-cover} (applied to $S=\Gamma_n([-\pi,\pi])$) we obtain
\begin{equation}\label{eq:N-lower-diam}
N_n(\varepsilon)\ge \frac{\diam(\Gamma_n)}{\sqrt{2}\,\varepsilon}.
\end{equation}

We now lower bound $\diam(\Gamma_n)$ by evaluating the curve at two parameter values.
Since $F_n(0)=\sum_{p\le n}v_p(n!)$ and $e^{ip\pi}=1$ for $p=2$ and $e^{ip\pi}=-1$ for odd primes,
\[
F_n(\pi)=v_2(n!)-\sum_{\substack{p\le n\\ p\ \mathrm{odd}}} v_p(n!).
\]
Therefore
\[
|F_n(0)-F_n(\pi)|
=
2\sum_{\substack{p\le n\\ p\ \mathrm{odd}}} v_p(n!),
\]
and since both points lie on the curve, $\diam(\Gamma_n)\ge |F_n(0)-F_n(\pi)|$.
Using Lemma \ref{lem:vp-bounds},
\[
\sum_{\substack{p\le n\\ p\ \mathrm{odd}}} v_p(n!)
\ge
\sum_{\substack{p\le n\\ p\ \mathrm{odd}}} \left(\frac{n}{p}-1\right)
=
n\sum_{\substack{p\le n\\ p\ \mathrm{odd}}}\frac{1}{p}
-\big(\pi(n)-1\big).
\]
By Mertens' theorem (see, e.g., \cite[Ch.~1]{Montgomery1994}),
\[
\sum_{p\le n}\frac{1}{p}=\log\log n + O(1),
\]
so $n\sum_{p\le n,\,p\ \mathrm{odd}}1/p \gtrsim n\log\log n$, while $\pi(n)=O(n/\log n)$ by the prime number theorem.
Hence for all large $n$,
\[
\diam(\Gamma_n)\ge c_1\,n\log\log n
\]
for some $c_1>0$.
Combining with \eqref{eq:N-lower-diam} gives the stated bound.
\end{proof}

\section{Discussion and conclusion}

The results above establish several rigorous obstructions to uniform geometric regularity in the family $\{\Gamma_n\}$. 
We have shown that the curve lengths diverge as $n\to\infty$, so there is no uniform bound on length within the family. 
Moreover, the first derivatives are not uniformly bounded, and therefore no subsequence can admit a uniform $C^1$ bound on $[-\pi,\pi]$. 
The second derivatives are likewise unbounded, excluding any uniform $C^{1,1}$ control. 
In addition, the diameter grows at least on the order of $n\log\log n$, which in turn forces explicit lower bounds on the covering numbers.

Together, these results provide a rigorous explanation for why the curves studied experimentally in \cite{Vartziotis2026a,Vartziotis2026b} cannot stabilize into uniformly smooth geometric objects as $n$ increases.

No claim is made regarding the existence of a limiting fractal dimension. For background on fractal dimension of curves, see \cite{Falconer2014,Tricot1995}.

\section*{Acknowledgments}

The author thanks NIKI Digital Engineering and TWT GmbH Science \& Innovation for support. He also thanks S. Katsioli, and V. Maroulis for helpful discussions.

\end{document}